\documentclass[12pt]{article}
\usepackage[cp1251]{inputenc}
\usepackage{amssymb,amsmath}
\usepackage{mathrsfs}
\usepackage{array, amsfonts, mathrsfs}
\usepackage{amssymb}

\usepackage{amsthm}
\usepackage{graphics, graphicx}
\usepackage[colorlinks=true]{hyperref}
\usepackage{float}
\usepackage{enumitem}

\newtheorem{Theorem}{Theorem}
\newtheorem{Lemma}{Lemma}

\newtheorem{Convention}{Convention}

\title{Wave propagation in abstract dynamical system with boundary control}
\author{M.I.Belishev\thanks {St.Petersburg Department of Steklov Mathematical Institute, St.Petersburg, Russian Federation,
        \newline
        e-mail: belishev@pdmi.ras.ru,
        \newline
        ORCID: 0000-0002-4759-7428;
        }.}

\date{}

\begin{document}
\maketitle

\begin{abstract}
Let $L_0$ be a positive definite ope\-ra\-tor in a Hilbert space
$\mathscr H$ with the defect indexes $n_\pm\geqslant 1$ and let
$\{{\rm Ker\,}L^*_0;\Gamma_1,\Gamma_2\}$ be its canonical (by
M.I.Vishik) boundary triple. The paper deals with an evolutionary
dynamical system of the form
\begin{align*}
& u_{tt}+{L_0^*} u=0 &&\text{in}\,\,{\mathscr H},\,\,\,t>0;\\
& u\big|_{t=0}=u_t\big|_{t=0}=0 && {\rm in}\,\,{\mathscr H};\\
& \Gamma_1 u=f(t), && t\geqslant 0,
\end{align*}
where $f$ is a boundary control (a ${\rm Ker\,}L^*_0$-valued
function of time), $u=u^f(t)$ is a trajectory.  Some of the
general properties of such systems are considered. An abstract
analog of the finiteness principle of wave propagation speed is
revealed.
\end{abstract}

\noindent{\bf Key words:}\,\,\, symmetric semi-bounded operator,
Vishik boundary triple, dynamic system with boundary control,
finiteness of wave propagation speed.

\noindent{\bf MSC:}\,\,\,35Lxx, 35L05, 35Q93, 47B25.

\bigskip

\rightline{\bf Dedicated to the 85-th jubilee of
A.S.Blagoveshchenskii}

\section{About the paper}
\noindent$\bullet$\,\,\, A dynamical system with boundary control
(DSBC) that we deal with, is determined by a symmetric
semi-bounded operator with nonzero defect indexes. We are
interested in most general properties of such systems. Motivation
comes from a program of constructing a functional model of such
operators (the so-called {\it wave model}: see \cite{B JOT},
\cite{BSim_1}-\cite{BSim_4}). The given paper develops the results
\cite{B DSBC IP 2001,BD_DSBC} on the general properties of DSBC.
Perhaps, most curious of new facts is that the finiteness
principle of wave propagation speed (for short, FS principle),
which is well known and holds in numerous applications, does have
a relevant analog for abstract DSBC.
\smallskip

\noindent$\bullet$\,\,\, The paper is dedicated to the jubilee of
my teacher Aleksandr Sergeevich Blagoveshchenskii, one of the
pioneers and creators of the dynamical inverse problems theory. At
one time, he explained me the deepness and opportunities of the
D'Alembert formula. It would not be an exaggeration to say that
the given work is done in the manner and technique of Alexander
Sergeevich.

\section{Operator $L_0$}

\subsubsection*{Boundary triple}
\noindent$\bullet$\,\,\, As was noted above, DSBC is associated
with a semi-bounded operator. The class of these operators that we
deal with, is the following. We assume that $L_0$ is a closed
densely defined symmetric positive definite operator in a Hilbert
space ${\mathscr H}$ with nonzero defect indexes; so that
$$
\overline{{\rm Dom\,} L_0}={\mathscr H};\quad
L_0\subset{L_0^*};\quad L_0\geqslant\gamma\,\mathbb
I,\,\,\,\,\gamma>0;\quad 1\leqslant
n_+^{L_0}=n_-^{L_0}\leqslant\infty
$$
holds, where $\mathbb I$ is the identity operator. Note that by
virtue of $n_\pm^{L_0}\not=0$ such an operator is necessarily
unbounded.
\smallskip

\noindent$\bullet$\,\,\, We denote ${\mathscr K}:={\rm Ker\,}
{L_0^*}$ and use the (orthogonal) projection $P$ in ${\mathscr H}$
on ${\mathscr K}$. Note that ${\rm dim\,}{\mathscr K}
=n_\pm^{L_0}$ holds.

Let $L$ be the extension of $L_0$ by Friedrichs: $L_0\subset
L=L^*\subset L^*_0$, $L\geqslant\gamma\,\mathbb I$, ${\rm Ran\,}
L={\mathscr H}$. Its inverse $L^{-1}$ is a self-adjoint bounded
operator in ${\mathscr H}$.

The well-known decomposition by Vishik \cite{Vishik} is
\begin{equation}\label{Eq Vishik Decomp Dom}
{\rm Dom\,}L_0^*={\rm Dom\,}L_0 \overset{.}+L^{-1}{\mathscr
K}\overset{.}+{\mathscr K}={\rm Dom\,} L\overset{.}+{\mathscr K};
\end{equation}
the latter equality is established in the framework of M.Krein's
theory \cite{MMM}. Thus, each $y \in {\rm Dom\,}L_0^*$ is uniquely
represented in the form
\begin{equation}\label{Eq Vishik Decomp}
y=y_0+L^{-1}g+h =y'+h
\end{equation}
with some $g, h\in{\mathscr K}$ and $y':=y_0+L^{-1}g\in {\rm
Dom\,} L$. The components are determined by $y$ as follows:
\begin{equation}\label{Eq Vishik Decomp components}
y'=L^{-1}{L_0^*} y,\quad y_0=L^{-1}{L_0^*}(y-y'),\quad h=y-y'-y_0.
\end{equation}

The operators
$$
\Gamma_1:=L^{-1}{L_0^*}-\mathbb I,\quad\Gamma_2:=P{L_0^*};\qquad
{\rm Dom\,}\Gamma_{1,2}={\rm Dom\,}{L_0^*}
$$
are called {\it boundary operators}. By definitions and (\ref{Eq
Vishik Decomp}),
\begin{equation}\label{Eq Gamma1,2 y}
\Gamma_1y=-h, \qquad \Gamma_2 y =g.
\end{equation}
Also, these definitions imply
\begin{equation}\label{Eq Ran Gamma 1,2}
{\rm Ran\,}\Gamma_1={\rm Ran\,}\Gamma_2={\mathscr K}.
\end{equation}
Note that, in general, boundary operators may be unclosable;
moreover, such a situation is typical in applications. However, if
one endows ${\rm Dom\,}{L_0^*}$ with the graph-norm $\|y\|^2_{\rm
graph}=\|y\|^2+\|{L_0^*} y\|^2$ then $\Gamma_{1,2}$ become
continuous \cite{MMM}.
\smallskip

\noindent$\bullet$\,\,\, The relation
\begin{equation}\label{Eq Green}
({L_0^*} u,v)-(u,{L_0^*}
v)=(\Gamma_1u,\Gamma_2v)-(\Gamma_2u,\Gamma_1v),\qquad u,v\in{\rm
Dom\,}{L_0^*}
\end{equation}
is valid (see, e.g., \cite{BD_DSBC}). By operator theory
terminology \cite{MMM}, relations (\ref{Eq Ran Gamma 1,2}) and
(\ref{Eq Green}) mean that the collection $\{{\mathscr
K};\Gamma_1,\Gamma_2\}$ constitutes the {\it boundary triple} of
the operator $L_0$. The general boundary triple theory provides
\begin{equation}\label{Eq L0,L}
L_0={L_0^*}\upharpoonright[{\rm Ker\,}\Gamma_1\cap{\rm
Ker\,}\Gamma_2],\qquad L={L_0^*}\upharpoonright{\rm Ker\,}\Gamma_1
\end{equation}
(see \cite{MMM}, Chapter 7).
\smallskip

\noindent$\bullet$\,\,\, A possible way to realize decomposition
(\ref{Eq Vishik Decomp}) is to solve two "boundary value
problems".
\begin{Lemma}\label{L abstract BVP}
Let $y=y_0+L^{-1}g+h$ be the Vishik decomposition of $y\in{\rm
Dom\,}{L_0^*}$. Then the elements $h$ and $g$ are uniquely
determined by the relations
\begin{equation}\label{Eq abstract BVP 1}
{L_0^*} h=0; \qquad \Gamma_1 h =\Gamma_1 y.
\end{equation}
and
\begin{equation}\label{Eq abstract BVP 2}
{L_0^*}^2 w=0; \quad \Gamma_1 w=0;\quad \Gamma_2 w=\Gamma_2(y-h),
\end{equation}
respectively, where $w:=L^{-1}g$.
\end{Lemma}
\begin{proof}
Element $h$ obeying the second relation in (\ref{Eq abstract BVP
1}), does exist due to (\ref{Eq Ran Gamma 1,2}). It is unique.
Indeed, if $h'$ satisfies (\ref{Eq abstract BVP 1}) then $\tilde
y:=h-h'$ obeys $\Gamma_1 \tilde y=0$, i.e., $\tilde y\in {\rm
Dom\,} L$. The latter implies $\tilde y=0$ by virtue of ${\rm
Dom\,} L\cap{\rm Ker\,}{L_0^*}=\{0\}$ (see (\ref{Eq Vishik Decomp
Dom})).

Since ${L_0^*}{L_0^*} L^{-1}g={L_0^*} g=0$, $L^{-1}g\in {\rm
Dom\,} L$ is valid (so that $\Gamma_1L^{-1}g\overset{{\rm
see\,}(\ref{Eq L0,L})}=0$ and
$\Gamma_2L^{-1}g=\Gamma_2(y-y_0-h)\overset{(\ref{Eq
L0,L})}=\Gamma_2(y-h)$ hold), we see that $w=L^{-1}g$ solves
problem (\ref{Eq abstract BVP 2}). If $w'$ also solves it, for
$\tilde w:=w-w'$ one has $\Gamma_1\tilde w=\Gamma_2 \tilde w=0$
that leads to $\tilde w\in{\rm Dom\,} L_0$ by virtue of (\ref{Eq
L0,L}). Therefore, ${L_0^*} \tilde w= L_0\tilde w\in{\rm Ran\,}
L_0$ and, hence, ${L_0^*}\tilde w\bot{\rm Ker\,}{L_0^*}$. The
latter makes ${L_0^*}{L_0^*}\tilde w=0$ possible only if
${L_0^*}\tilde w=0$. Since ${L_0^*} \tilde w=L_0\tilde w=0$, we
arrive at $\tilde w=0$ by injectivity of $L_0$.
\end{proof}

Thus, to determine $h$ and $g$ in (\ref{Eq Vishik Decomp}), one
can find $h$ from (\ref{Eq abstract BVP 1}), solve (\ref{Eq
abstract BVP 2}) and then get $g=Lw$.

\subsubsection*{Example}

\noindent$\bullet$\,\,\, As an illustration, we consider the
Laplace operator. Let $(\Omega,g)$ be a compact smooth\footnote{In
the subsequent, {\it smooth} always means $C^\infty$-smooth.}
Riemannian manifold of dimension $n\geqslant 2$ \footnote{the case
$\Omega\subset\mathbb R^n$ is quite suitable for our goals.} with
the smooth connected boundary $\Gamma$, let $\Delta$ be the
Laplace-Beltrami differential operator in $\Omega$.

Let $H^p(\Omega), \,\,p=1,2$, $H^1_0(\Omega)=\{y\in
H^1(\Omega)\,|\,\,y\big|_{\Gamma}=0\}$ and $H^2_0(\Omega)=\{y\in
H^2(\Omega)\,|\,\,y\big|_{\Gamma}=\partial_\nu
y\big|_{\Gamma}=0\}$ be the Sobolev spaces ($\nu$ is the outward
normal on $\Gamma$). We put ${\mathscr H}:=L_2(\Omega)$ and denote
by ${\rm Harm}(\Omega):=\{h\in{\mathscr H}\,|\,\,\Delta
h=0\,\,{\rm in}\,\,\Omega\setminus\Gamma\}$ the subspace of
harmonic functions. The following is the well-known facts.

The operator ({\it minimal Laplacian})
$L_0:=\overline{-\Delta{\upharpoonright}
C^\infty_0(\Omega)}=-\Delta{\upharpoonright} H^2_0(\Omega)$ is
positive definite. Its adjoint ({\it maximal Laplacian}) is
${L_0^*}=-\Delta{\upharpoonright}[H^2(\Omega)+{\rm
Harm}(\Omega)]$, and ${\mathscr K}={\rm Ker\,}{L_0^*}={\rm
Harm}(\Omega)$ holds. The Friedrichs extension of $L_0$ is
$L=-\Delta{\upharpoonright}[H^2(\Omega)\cap H^1_0(\Omega)]$.
\smallskip

\noindent$\bullet$\,\,\, By (\ref{Eq Vishik Decomp}) and (\ref{Eq
Gamma1,2 y}), to describe how the boundary operators
$\Gamma_{1,2}$ act, one needs to show how to find the harmonic
functions $h$ and $g$ for a given $y\in{\rm Dom\,}{L_0^*}$. Since
$y_0$ and $L^{-1}g$ belong to $H^1_0(\Omega)$, we have $y=h$ on
$\Gamma$. Thus, the function $h$ can be specified as the (unique)
solution of the Dirichlet problem
\begin{equation}\label{Eq concrete BVP 1}
\Delta h=0 \quad{\rm in}\,\,\,\Omega\setminus\Gamma;\qquad
h=y\quad{\rm on}\,\,\,\Gamma.
\end{equation}
To find the harmonic $g$, we recall that the summand $y_0$ in
(\ref{Eq Vishik Decomp}) belongs to ${\rm Dom\,}
L_0=H^2_0(\Omega)$ and, hence, obeys $\partial_\nu
y_0\big|_\Gamma=0$. This implies
$$
\partial_\nu L^{-1}g\overset{(\ref{Eq Vishik
Decomp})}=\partial_\nu(y-y_0-h)=\partial_\nu(y-h).
$$
In the mean time, we have $L^{-1}g\in H^1_0(\Omega)$  and
${L_0^*}{L_0^*} L^{-1}g={L_0^*} g=0$. As a result, $L^{-1}g$ obeys
\begin{equation}\label{Eq concrete BVP 2}
\Delta^2 (L^{-1}g)=0 \quad{\rm
in}\,\,\,\Omega\setminus\Gamma;\quad
L^{-1}g=0,\,\,\,\partial_\nu(L^{-1}g)=\partial_\nu(y-h)\,\,\,{\rm
on}\,\,\,\Gamma
\end{equation}
($h$ is already known). Solving this well-posed Cauchy problem for
the biharmonic equation, we get $L^{-1}g$ and then find $g=\Delta
L^{-1}g$.
\smallskip

As is easy to recognize, (\ref{Eq concrete BVP 1}) and (\ref{Eq
concrete BVP 2}) are some concrete versions of (\ref{Eq abstract
BVP 1}) and (\ref{Eq abstract BVP 2}) respectively. We get rights
to claim that ${L_0^*} h=0$ and ${L_0^*}^2 w=0$ are the abstract
Laplace and biharmonic equations respectively.

\section{Dynamics determined by operator $L_0$}

\subsubsection*{System $\alpha$}

\noindent$\bullet$\,\,\, The boundary triple, in turn, determines
a {\it dynamical system with boundary control} (DSBC)  of the form
\begin{align}
\label{Eq 1}& {\ddot u}+L_0^*u = 0  && {\rm in}\,\,\,{{\mathscr H}}, \,\,\,t>0;\\
\label{Eq 2}& u|_{t=0}={\dot u}|_{t=0}=0 && {\rm in}\,\,\,{{\mathscr H}};\\
\label{Eq 3}& \Gamma_1 u = f && {\rm in}\,\,\,{{\mathscr
K}},\,\,\,t \geqslant 0,
\end{align}
where $\dot{(\,\,)}:=\frac{d}{dt}$; $f=f(t)$ is a ${\mathscr
K}$-valued function of time ({\it boundary control}). By
$u=u^f(t)$ we denote the solution ({\it wave}). For short, we call
(\ref{Eq 1})--(\ref{Eq 3}) just system $\alpha$; function
$u^f(\cdot)$ is its {\it trajectory}.

Introduce the class of smooth controls
$$
{\mathscr M}:=\{f\in C^\infty([0,\infty);{\mathscr K})\,|\,\,{\rm
supp\,}f\subset(0,\infty)\},
$$
vanishing near $t=0$. As is shown in \cite{BD_DSBC}, for each
$f\in{\mathscr M}$ there exists a unique classical solution
$u=u^f(t)\in  C^\infty([0,\infty);{\mathscr H})$ (a {\it smooth
wave}) obeying $u^f(t)\in{\rm Dom\,}{L_0^*},\,\,\,t\geqslant 0$
and represented in the form
\begin{equation}\label{Eq u^f}
u^f(t)=-f(t)+L^{-{1\over 2}}\int_0^t\sin\,[(t-s)L^{1\over
2}]\,{\ddot f}(s)\,ds,\qquad t\geqslant 0.
\end{equation}
Integrating by parts, we get the equivalent representation
\begin{equation}\label{Eq u^f+}
u^f(t)=-f(t)+L^{-1}\int_0^t\left\lbrace1-\cos\,[(t-s)L^{1\over
2}]\right\rbrace\,\overset{...}{f}(s)\,ds=-f(t)+u^f_L(t),\quad
t\geqslant 0,
\end{equation}
where the summand $u^f_L(t)\in{\rm Dom\,} L$ corresponds to the
element $y'\in{\rm Dom\,} L$ in decomposition (\ref{Eq Vishik
Decomp}).

The right hand side in (\ref{Eq u^f}) is meaningful for any $f\in
H^2_{\rm loc}([0,\infty); {\mathscr K})$ vanishing near $t=0$. For
such controls, it defines a (generalized) solution $u^f$ to
(\ref{Eq 1})--(\ref{Eq 3}). The following is some of its
properties.
\smallskip

\noindent$\bullet$\,\,\, In what follows, we assume that all time
functions are extended to $t< 0$ {\it by zero}.

Let a map ${\cal T}_\tau$ delay functions by the rule $({\cal
T}_\tau w)(t):=w(t-\tau)$. Since the operator ${L_0^*}$ which
governs the evolution of system $\alpha$, does not depend on time,
the steady state relation
\begin{equation}\label{Eq steady state}
u^{{\cal T}_\tau f}(t)=({\cal T}_\tau u^f)(t)=u^f(t-\tau),\qquad
t\geqslant 0,\,\,\,\tau>0
\end{equation}
and its consequence
\begin{equation}\label{Eq steady state+}
u^{\dot f}(t)= {\dot u}^f(t),\quad\,u^{\ddot f}(t)= {\ddot
u}^{f}(t)\overset{(\ref{Eq 1})}=-L_0^*u^f(t),\qquad t\geqslant 0
\end{equation}
hold.
\smallskip

\noindent$\bullet$\,\,\, The following is the control theory
attributes of system $\alpha$.
\smallskip

The space of controls (inputs) $\mathscr F:=L_2^{\rm loc}(\mathbb
R_+;\mathscr K)$ is an {\it external space}. It contains a family
of delayed controls $\mathscr F_\tau:=\mathcal T_\tau\mathscr
F,\,\,\,\tau>0$ and the smooth class $\mathscr M$, which satisfies
\begin{equation}\label{Eq M}
\frac{d^p}{dt^p}\mathscr M=\mathscr M,\qquad p=1,2,\dots
\end{equation}
and is locally dense in $\mathscr F$:
$\overline{\{f\big|_{[0,T]}\,|\,\,f\in\mathscr
M\}}=L_2([0,T];\mathscr K),\qquad T>0$.
\smallskip

The space $\mathscr H$ is an {\it internal space}, the waves
(states) $u^f(t)$ are its elements. It contains an increasing
family of {\it reachable sets} $\mathscr
U^\tau:=\{u^f(\tau)\,|\,\,f\in\mathscr M\},\,\,\,\tau\geqslant 0$
and the total reachable set $\mathscr U:={\rm span\,}\{\mathscr
U^\tau\,|\,\,\tau\geqslant 0\}$. By (\ref{Eq M}) and (\ref{Eq
steady state+}) an invariance of reachable sets
\begin{equation}\label{Eq invar U tau}
L_0^*\mathscr U^\tau=\mathscr U^\tau,\,\,\, \tau \geqslant
0;\qquad L_0^*\mathscr U=\mathscr U
\end{equation}
holds.
\smallskip

In system $\alpha$, the input-state map is
$$
W^\tau:\mathscr F\to {\mathscr H},\,\,{\rm Dom\,}W^\tau=\mathscr
M,\,\,\,\,W^\tau f:=u^f(\tau),\qquad \tau\geqslant 0.
$$
\begin{Lemma}\label{L closability W}
For any $\tau>0$, the map $W^\tau$ is closable.
\end{Lemma}
\begin{proof}
$\bf 1.$\,\,\,Take an arbitrary $\phi\in{\mathscr H}$. Applying
$L^{-{1\over2}}$ in (\ref{Eq u^f}), we have
\begin{equation}\label{Eq aux 1}
(L^{-{1\over2}}u^f(\tau),\phi)=-(L^{-{1\over2}}f(\tau),\phi)+
(L^{-1}\int_0^\tau\sin[(\tau-s)L^{{1\over2}}]{\ddot
f}(s)\,ds,\phi)=:-I+I\!I.
\end{equation}
Since $L=L^*$, the second summand is of the form
\begin{align*}
& I\!I= \int_0^\tau ds\,(\ddot f(s),\psi(s))
\end{align*}
with $\psi(s):=L^{-1}\sin[(\tau-s)L^{1\over2}]\phi$ obeying
$\ddot\psi(s)=-\sin[(\tau-s)L^{1\over2}]\phi \in {\mathscr H}$.
Then one can easily justify integration by parts:
\begin{align*}
& I\!I=\left[(\dot
f(s),L^{-1}\sin[(\tau-s)L^{{1\over2}}]\phi)+(f(s),L^{-{1\over2}}\cos[(\tau-s)L^{{1\over2}}]\phi)
\right]\bigg|_{s=0}^{s=\tau}+\\
&+\int_0^\tau
(f(s),\sin[(\tau-s)L^{{1\over2}}]\phi)\,ds=(f(\tau),L^{-{1\over2}}\phi)+\\
& +\int_0^\tau
(f(s),\sin[(\tau-s)L^{{1\over2}}]\phi)\,ds=(L^{-{1\over2}}f(\tau),\phi)+\int_0^\tau
(\sin[(\tau-s)L^{{1\over2}}]f(s),\phi)\,ds.
\end{align*}
Substituting $I\!I$ in the form of the latter sum to (\ref{Eq aux
1}) and taking into account the arbitrariness of $\phi$, one
represents
\begin{equation}\label{Eq aux 2}
L^{-{1\over2}}u^f(\tau)=\int_0^\tau
\sin[(\tau-s)L^{{1\over2}}]f(s)\,ds.
\end{equation}

$\bf 2.$\,\,\,In view of (\ref{Eq u^f}), the value of the wave
$u^f(\tau)$ is determined by the values $f\big|_{0\leqslant
t\leqslant\tau}$ of  control (does not depend on
$f\big|_{t>\tau}$). Let $f\big|_{[0,\tau]}\to 0$ in
$L_2([0,\tau];{\mathscr K})$ and $u^f(\tau)\to y$ in ${\mathscr
H}$. The limit passage in (\ref{Eq aux 2}) leads to
$L^{-{1\over2}}y=0$. The latter implies $y=0$.
\smallskip

Thus, $f\to 0$ and $W^\tau f\to y$ yields $y=0$, i.e., $W^\tau$ is
closable.
\end{proof}

Denote $\mathscr F^\tau:=L_2([0,\tau];{\mathscr K})$ and $\mathscr
M^\tau:=\{f\big|_{[0,\tau]}\,|\,\,f\in\mathscr M\}$. By (\ref{Eq
u^f}), the value $u^f(\tau)$ of the wave is determined by the
values $f\big|_{0\leqslant t\leqslant\tau}$ of control (does not
depend on $f\big|_{t>\tau}$). Therefore, the map
$$
W^\tau:\mathscr F^\tau\to {\mathscr H},\,\,\,{\rm
Dom\,}W^\tau=\mathscr M,\,\,\,\,W^\tau f:=u^f(\tau)
$$
is well defined and closable for all $\tau>0$.

\subsubsection*{Example}

\noindent$\bullet$\,\,\, In the example chosen above as
illustration, we have ${\mathscr K}={\rm Harm\,}(\Omega)$, whereas
the bijection ${\rm Harm\,}(\Omega)\ni y \leftrightarrow
y\big|_\Gamma$ does occur. Therefore, by (\ref{Eq Gamma1,2 y}),
system $\alpha$ can be written in the equivalent (traditional)
form of an initial - boundary value problem
\begin{align}
\label{Eq 1+}& u_{tt}-\Delta u = 0  && {\rm in}\,\,\,(\Omega\setminus\Gamma)\times \mathbb R_+;\\
\label{Eq 2+}& u|_{t=0}=u_t|_{t=0}=0 && {\rm in}\,\,\,\Omega;\\
\label{Eq 3+}& u = f && {\rm
on}\,\,\,\Gamma\times\overline{\mathbb R_+},
\end{align}
which describes propagation of wave $u=u^f(x,t)$ in $\Omega$, the
wave being initiated by the boundary source (control)
$f=f(\gamma,t)$. Let us recall some of its known properties.


For smooth controls of the class ${\mathscr M}:=\{f\in
C^\infty(\Gamma\times\overline{\mathbb R_+})\,|\,\,{\rm
supp\,}f\subset\Gamma\times\mathbb R_+\}$ vanishing near $t=0$,
system (\ref{Eq 1+})--(\ref{Eq 3+}) has a unique classical
solution $u^f\in C^\infty(\Omega\times\overline{\mathbb R_+})$.

For all $\tau>0$, the map $W^\tau:f\big|_{[0,\tau]}\mapsto
u^f(\cdot,\tau)$ is continuous from ${\mathscr
F}^\tau:=L_2(\Gamma\times[0,\tau])$ to ${\mathscr H}$ \,
\cite{LTrDynamReported}.
\smallskip

Let $\Omega^r:=\{x\in\Omega\,|\,\,{\rm dist}(x,\Gamma)<r\}$ be a
metric neighborhood of the boundary of radius $r>0$; denote
$T_*:=\inf\,\{r>0\,|\,\,\Omega^r=\Omega\}$. The relation
\begin{equation}\label{Eq supp u^f}
{\rm supp\,}u^f(\cdot,t)\subset\overline{\Omega^t},\qquad t>0
\end{equation}
holds and shows that waves move from the boundary into the
manifold with velocity $\leqslant 1$. At the moment $t=T_*$ the
waves fill up the whole $\Omega$. In what follows we refer to
these facts as a {\it finiteness of wave propagation speed}
principle (FS principle). It corresponds to a hyperbolicity of the
initial boundary value problem (\ref{Eq 1+})--(\ref{Eq 3+}).
\smallskip

\noindent$\bullet$\,\,\, Introduce the reachable sets ${\mathscr
U}^t:=\{u^f(\cdot,t)\,|\,\,f\in{\mathscr M}\}$ and note the
evident relation ${\mathscr U}^s\subset{\mathscr U}^t$ for $s<t$.
As a consequence of (\ref{Eq supp u^f}), we have the embedding
${\mathscr U}^\tau\subset{\mathscr H}^\tau:=\{y\in{\mathscr
H}\,|\,\,{\rm supp\,}y\subset\overline{\Omega^\tau}\}$. A
remarkable fact, which is interpreted as a local approximate
boundary controllability of system $\alpha$, is that this
embedding is dense: the equality
\begin{equation}\label{Eq U subset H}
\overline{{\mathscr U}^\tau}\,=\,{\mathscr H}^\tau,\qquad \tau>0
\end{equation}
holds\, \cite{B EACM}, \cite{Tat}. Since $\Omega$ is compact, for
$\tau\geqslant T_*$ relation (\ref{Eq U subset H}) implies
$\overline{{\mathscr U}^\tau}={\mathscr H}$, so that system
(\ref{Eq 1+})--(\ref{Eq 3+}) is controllable for any moment
$T>T_*$.
\smallskip

Note in addition that, in the given Example, the family of the
projectors $P^\tau$ in $\mathscr H$ onto $\overline{{\mathscr
U}^\tau}$ is continuous with respect to $\tau$. However, in the
system governed by the Maxwell equations, the relevant family may
have infinite-dimensional breaks: see \cite{Demch}. Thus, the
continuity of $\overline{{\mathscr U}^\tau}$ is definitely not a
general fact.

\subsubsection*{Controllability of system $\alpha$}

\noindent$\bullet$\,\,\, For the abstract system $\alpha$ of the
form (\ref{Eq 1})--(\ref{Eq 3}) ant its reachable sets ${\mathscr
U}^\tau=\{u^f(\tau)\,|\,\,f\in{\mathscr M}\}$, to discuss property
(\ref{Eq U subset H}) is meaningless because there is no analog of
the subspaces ${\mathscr H}^\tau$. Nevertheless, the question can
be posed for the total reachable set $\mathscr U$ as follows. We
say that system $\alpha$ is controllable if the equality
\begin{equation}\label{Eq abstract contr}
\overline{{\mathscr U}}\,=\,{\mathscr H}
\end{equation}
holds. If (\ref{Eq abstract contr}) is not valid, we say
${\mathscr D}:={\mathscr H}\ominus\overline{\mathscr U}$ to be a
defect (unreachable) subspace. The question is on the conditions
which provide (\ref{Eq abstract contr}).  The answer is the
following.

Let $A$ be an operator in a Hilbert space ${\mathscr H}$ and
${\mathscr G}\subset{\mathscr H}$ be a subspace. We say that
${\mathscr G}$ is an {\it invariant subspace} of $A$ if
$\overline{{\rm Dom\,} A \cap{\mathscr G}}={\mathscr G}$ and
$A({\rm Dom\,} A \cap{\mathscr G})\subset{\mathscr G}$ hold
\footnote{see some comments to this definition in \cite{BSim_4}},
whereas operator $A_{\mathscr G}:{\mathscr G}\to{\mathscr
G},\,\,{\rm Dom\,} A_{\mathscr G}={\rm Dom\,} A\cap{\mathscr
G},\,\,A_{\mathscr G} y:=Ay$ is called a part of $A$ in ${\mathscr
G}$.

A symmetric densely defined operator $A$ is said to be {\it
completely non-self-adjoint} if it has no (substantially)
self-adjoint parts, i.e, there are no parts $A_{\mathscr G}$,
which satisfy $A_{\mathscr G}^*=\overline{A_{\mathscr G}}$.
\smallskip

System $\alpha$ is determined by operator $L_0$. As is shown in
\cite{BD_DSBC}, it is controllable, i.e., (\ref{Eq abstract
contr}) holds, if and only if $L_0$ is completely
non-self-adjoint.
\smallskip

\noindent$\bullet$\,\,\, There exist dynamical systems (\ref{Eq
1})--(\ref{Eq 3}), in which (\ref{Eq U subset H}) takes the form
$\overline{{\mathscr U}^\tau}={\mathscr H}$ for any $\tau>0$. As
example, one can mention the system on $\Omega\subset \mathbb R^n$
governed by the Euler-Bernoully equation $u_{tt}+\Delta^2u=0$ and
relevant boundary controls \cite{LT}. The following agreement
excludes such cases from consideration as trivial ones.
\begin{Convention}\label{Conv 1}
For system $\alpha$, we assume that there are $\tau>0$ such that
$\overline{{\mathscr U}^{\tau+\epsilon}}\ominus\overline{{\mathscr
U}^{\tau}}\not=\{0\}$ holds for any $\epsilon>0$.
\end{Convention}
In other words, it is assumed that the family of the reachable
subspaces $\overline{{\mathscr U}^{\tau}}$ does have the positive
growth points.

In the concrete systems realizing the trivial case, the FS
principle is not in force: (\ref{Eq supp u^f}) is broken and the
waves propagate with infinite velocity. If the condition in
Convention \ref{Conv 1} is violated, our further results remain
true but become trivial. For the rest of the paper, Convention
\ref{Conv 1} is accepted.

\subsubsection*{System $\beta$}

\noindent$\bullet$\,\,\, Consider a dynamical system $\beta$ of
the form
\begin{align}
\label{Eq beta 1}& \ddot v+Lv = \psi  && {\rm in}\,\,\,{{\mathscr H}}, \,\,\,t>0;\\
\label{Eq beta 2}& v|_{t=0}=\dot v|_{t=0}=0 && {\rm
in}\,\,\,{{\mathscr H}},
\end{align}
controlled by a {\it source} $\psi$, which is an ${\mathscr
H}$-valued function of time. By $v=v^\psi(t)$ we denote the
solution. For smooth sources of the class
$$
{\mathscr N}:=\left\{\psi\in C^\infty([0,\infty);{\mathscr
H})\,\big|\,\,{\rm supp\,} \psi\subset(0,\infty)\right\}
$$
vanishing near $t=0$, the solution is unique, classical, smooth,
is represented in the form
\begin{equation}\label{Eq v^psi repres general}
v^\psi(t)=L^{-{1\over 2}}\int_0^t\sin[(t-s)L^{{1\over
2}}]\,\psi(s)\,ds,\qquad t\geqslant 0
\end{equation}
and belongs to ${\rm Dom\,} L$ for any $t$. By (\ref{Eq L0,L}),
the latter implies
\begin{equation}\label{Eq Gamma1 v^psi=0}
\Gamma_1 v^\psi(t)=0,\qquad t\geqslant0.
\end{equation}
The right hand side of (\ref{Eq v^psi repres general}) makes sense
for $\psi\in  L^{\rm loc}_2([0,\infty);{\mathscr H})$ and is
referred to as a generalized solution. If $\psi\in H^1_{\rm
loc}([0,\infty);{\mathscr H})$ then integration by parts in
(\ref{Eq v^psi repres general}) provides
\begin{equation*}
v^\psi(t)=L^{-1}\int_0^t\left\{\mathbb I-\cos[(t-s)L^{{1\over
2}}]\right\}\,\dot\psi(s)\,ds,\qquad t\geqslant0.
\end{equation*}
In this case, $v^\psi(t)\in{\rm Dom\,} L$ holds and relation
(\ref{Eq Gamma1 v^psi=0}) remains valid.
\smallskip

\noindent$\bullet$\,\,\, As a partial case, we deal with the
instantaneous sources $\psi=\delta(t)y$, where $y\in{\mathscr H}$
and $\delta(\cdot)$ is the Dirac delta-function, and the
corresponding solutions $v=v^{\delta y}(t)=:v^y(t)$ of the system
\begin{align}
\label{Eq beta 1'}& \ddot v+Lv = 0  && {\rm in}\,\,\,{{\mathscr H}}, \,\,\,t>0;\\
\label{Eq beta 2'}& v|_{t=0}=0,\,\,\, \dot v|_{t=0}=y && {\rm
in}\,\,\,{{\mathscr H}}.
\end{align}
The solution is represented in the form
\begin{equation}\label{Eq v^delta y repres general}
v^{y}(t)=L^{-{1\over 2}}[\sin (t\,L^{{1\over 2}})]\,y,\qquad
t\geqslant0.
\end{equation}
A generalized solution $v^y$ is well defined for any
$y\in{\mathscr H}$. If $y\in{\rm Dom\,} L^{{1\over 2}}$ then
$v^y(t)\in{\rm Dom\,} L$ and property (\ref{Eq Gamma1 v^psi=0})
holds for $v^y$:
\begin{equation}\label{Eq Gamma1 v^y=0}
\Gamma_1 v^y(t)=0,\qquad t\geqslant0.
\end{equation}
\smallskip

\noindent$\bullet$\,\,\, Let us consider the relations between
systems $\alpha$ and $\beta$.
\begin{Lemma}\label{L dual system}
For any $T>0$, the relation
\begin{equation}\label{Eq (W^T)*}
(u^f(T),y)=-\int_0^T(f(t),\Gamma_2 v^y(T-t))\,dt,\qquad f\in
{\mathscr M},\,\,y\in{\rm Dom\,}L
\end{equation}
is valid. 
\end{Lemma}
\begin{proof} By the choice of controls and
proper smoothness of the corresponding trajectories $u^f$ and
$v^y$, one can easily justify the following calculations:
\begin{align*}
& 0\overset{(\ref{Eq 1})}=\int_0^T(\ddot u^f(t)+{L_0^*}
u^f(t),v^{y}(T-t))\,dt=\int_0^T(\ddot
u^f(t),v^y(T-t))\,dt+\\
& + \int_0^T({L_0^*} u^f(t),v^y(T-t))\,dt=[(\dot
u^f(t),v^y(T-t))+(u^f(t),\dot v^y(T-t))]\big|_{t=0}^{t=T}+\\
& +\int_0^T(u^f(t),\ddot v^y(T-t))\,dt+\int_0^T({L_0^*}
u^f(t),v^y(T-t))\,dt=\\
& \overset{(\ref{Eq 2}),\,(\ref{Eq
2**})}=(u^f(T),y)+\int_0^T(u^f(t),\ddot
v^y(T-t))\,dt+\int_0^T({L_0^*} u^f(t),v^y(T-t))\,dt=\\
& \overset{(\ref{Eq Green})}=(u^f(T),y)+\int_0^T(u^f(t),\ddot
v^y(T-t)+{L_0^*} v^y(T-t))\,dt+\\
& +\int_0^T[(\Gamma_1 u^f(t),\Gamma_2 v^y(T-t))-(\Gamma_2
u^f(t),\Gamma_1 v^y(T-t))]\,dt=\\
&\overset{(\ref{Eq 3}),\,(\ref{Eq beta 1'}),\,(\ref{Eq Gamma1
v^y=0})}=(u^f(T),y)+\int_0^T(f(t),\Gamma_2 v^y(T-t))\,dt
\end{align*}
(we also use $L_0^*v^y=Lv^y$ by $L\subset L_0^*$). Comparing the
beginning with the end, we arrive at (\ref{Eq (W^T)*}). 
\end{proof}
One more relation between trajectories of systems $\alpha$ and
$\beta$ is the following. Quite analogous calculations starting
from the equality $0=\int_0^T(\ddot u^f(t)+{L_0^*} u^f(t),
v^\psi(T-t))\,dt$ lead to the relation
\begin{equation}\label{Eq (f,Gamma2 v psi)}
\int_0^T(u^f(t),\psi(T-t))\,dt=-\int_0^T(f(t),\Gamma_2v^\psi(T-t))\,dt,\,\,\,\,
f\in{\mathscr M},\,\,\psi\in{\mathscr N}.
\end{equation}
\smallskip

\noindent$\bullet$\,\,\, Let us derive a consequence of (\ref{Eq
(f,Gamma2 v psi)}). We say that a source $\psi$ acts from a
subspace ${\mathscr G}\subset{\mathscr H}$ if $\psi(t)\in
{\mathscr G}$ holds for all $t$.

Fix a positive $\sigma<T$; let $\psi$ act from
$\overline{{\mathscr U}}\ominus\overline{{\mathscr U}^\sigma}$
\footnote{Recall Convention \ref{Conv 1}\,!}. Also, let
$f\in\mathscr F_{T-\sigma}\cap{\mathscr M}$ be a delayed control.
For such a choice, by (\ref{Eq steady state}) we have
$u^f(t)\in{\mathscr U}^\sigma$ for all $0\leqslant t\leqslant T$
that obeys $(u^f(t),\psi(T-t))=0,\,\,\,0\leqslant t\leqslant T$.
Hence, the left integral in (\ref{Eq (f,Gamma2 v psi)}) vanishes
and we obtain
$$
\int_{T-\sigma}^T(f(t),\Gamma_2v^\psi(T-t))\,dt=\int_0^\sigma(f(T-t),\Gamma_2v^\psi(t))\,dt=0.
$$
As a result, by arbitrariness of $f$ we conclude that for any
smooth $\psi$ acting from $\overline{{\mathscr
U}}\ominus\overline{{\mathscr U}^\sigma}$ the relation
\begin{equation}\label{Eq (Gamma2 psi=0)}
\Gamma_2 v^\psi(t)=0,\qquad 0\leqslant t\leqslant\sigma,
\end{equation}
holds. It will be used later.

Quite analogously, by the use of (\ref{Eq (W^T)*}) for system
(\ref{Eq beta 1'})--(\ref{Eq beta 2'}),  the relation
$y\in\left[\overline{{\mathscr U}}\ominus\overline{{\mathscr
U}^\sigma}\right]\cap{\rm Dom\,}L$ implies
\begin{equation}\label{Eq (Gamma2 v^y=0)}
\Gamma_2 v^y(t)=0,\qquad 0\leqslant t\leqslant\sigma
\end{equation}
is derived.

\section{Abstract FS principle}

As was already mentioned, to discuss property (\ref{Eq supp u^f})
for abstract systems $\alpha$ and $\beta$  
is meaningless: there are no manifolds, domains, boundaries in
them. However, a remarkable fact is that a relevant analog of FS
principle for them 
does exist.
\smallskip

\noindent$\bullet$\,\,\, At first, let us turn to the Example and
clarify, which fact related to FS principle, we are going to
reveal in the abstract case.

The corresponding systems $\beta$ is
\begin{align*}
& v_{tt}-\Delta v = \psi  && {\rm in}\,\,\,(\Omega\setminus\Gamma)\times \mathbb R_+;\\
& v|_{t=0}=v_t|_{t=0}=0 && {\rm in}\,\,\,\Omega;\\
& v = 0 && {\rm on}\,\,\,\Gamma\times\overline{\mathbb R_+}.
\end{align*}
Fix $0<\sigma<\tau<T_*$. Assume that the source $\psi$ acts from
the subspace ${\mathscr H}^\tau\ominus{\mathscr H}^\sigma$, i.e.,
${\rm
supp\,}\psi(\cdot,t)=\Omega^\tau\setminus\Omega^\sigma,\,\,t>0$
holds. Thus, the source is located in a `layer'
$\Omega^\tau\setminus\Omega^\sigma$, which is separated from the
boundary $\Gamma$ by distance $\sigma$, whereas
$\Omega\setminus\overline{\Omega^\tau}$ is a nonempty open set. In
such a case, the wave $v^\psi$ propagates in both directions from
the layer with velocity 1. By the latter, at the moment $t>0$ it
is located in the bigger layer
$\Omega^{\tau+t}\setminus\Omega^{\sigma-t}$. I terms of subspaces,
this can be written as $v^\psi(t)\in {\mathscr
H}^{\tau+t}\ominus{\mathscr H}^{\sigma-t}$. It is the property
that has an abstract analog for system $\beta$.
\smallskip

\begin{minipage}[t]{\textwidth}
\centering
\begin{figure}[H]
    \includegraphics[width=30pc]{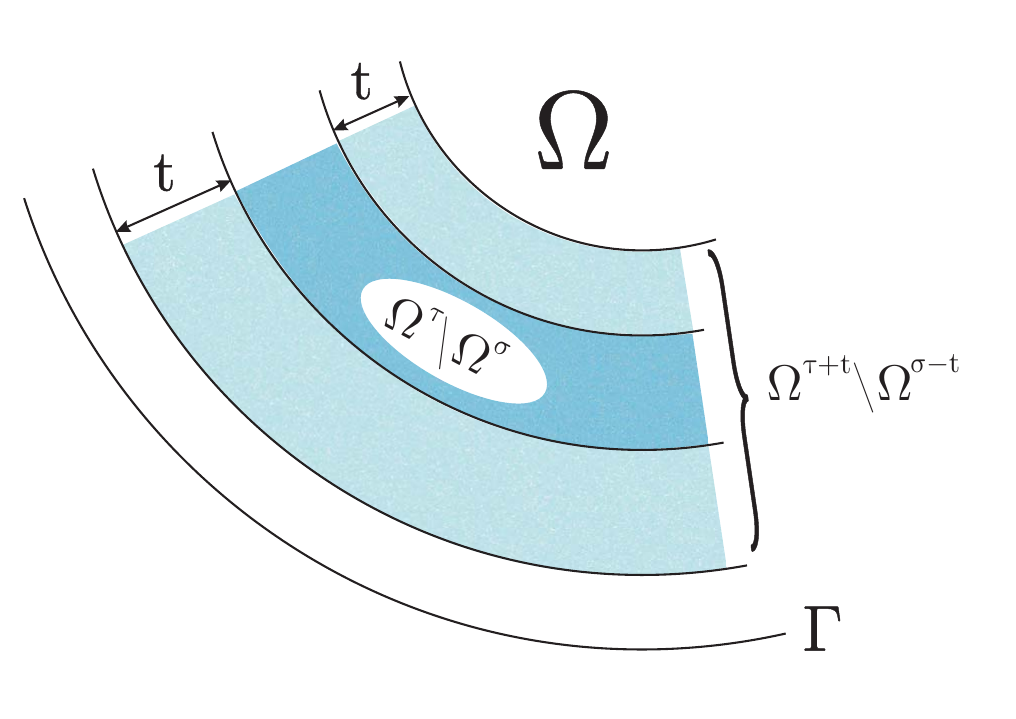}
    \caption{ Domains}
\end{figure}
\end{minipage}
\vspace{1pc}

\noindent$\bullet$\,\,\, The relevant analog is the following. For
convenience in formulation, we put ${\mathscr
U}^t\big|_{t<0}:=\{0\}$ and $\Psi:=L_2^{\rm loc}(\mathbb
R_+;\mathscr H)$.
\begin{Theorem}\label{T1}
Let \,$0<\sigma<\tau$ and let a source $\psi\in\Psi$ act from the
subspace $\overline{{\mathscr U}^\tau}\ominus\overline{{\mathscr
U}^\sigma}$. Then the relation $v^\psi(t)\in \overline{{\mathscr
U}^{\tau+t}}\ominus\overline{{\mathscr U}^{\sigma-t}}$ is valid
for all $t>0$.
\end{Theorem}
\begin{proof} The proof consists of a few steps.

\noindent{\bf Step 1.}\,\,\,Here we derive an auxiliary relation.
By
$$
C_{s,t}:=\{(\xi,\eta)\in \overline{\mathbb R^2_+}\,|\,\,0\leqslant
\eta\leqslant t,\,\,s-t+\eta\leqslant\xi\leqslant s+t-\eta\}
$$
we denote a characteristic cone of the string equation
$u_{tt}-u_{ss}=0$.

\begin{figure}
\centering
\begin{minipage}[t]{\textwidth}
\centering
    \includegraphics[width=30pc]{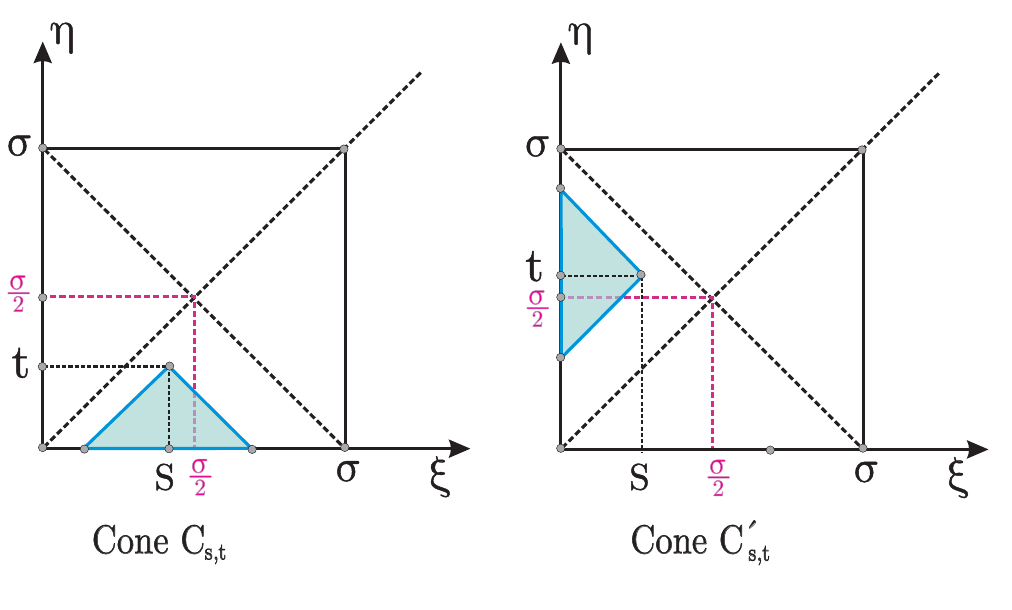}
    \caption{ Cones}
\end{minipage}
\end{figure}

\begin{Lemma}\label{L DAlembert 1}
Let $f\in{\mathscr M}$ and $\psi\in{\mathscr N}$ be a control and
a source in systems $\alpha$ and $\beta$. The relation
\begin{equation}\label{Eq DAlembert Cone C}
(v^\psi(s),u^f(t))=-{1\over 2}\int_{C_{s,t}}\left[(\Gamma_2
v^\psi(\xi),f(\eta))+(\psi(\xi),u^f(\eta))\right]\,d\xi\,d\eta,\quad
0\leqslant t\leqslant s
\end{equation}
is valid.
\end{Lemma}
\begin{proof}
For the Blagoveshchenskii function $b(s,t):=(v^\psi(s),u^f(t))$,
one has
\begin{align*}
& b_{tt}(s,t)-b_{ss}(s,t)=(v^\psi(s),\ddot u^f(t))-(\ddot
v^\psi(s),u^f(t))\overset{(\ref{Eq 1}),\,(\ref{Eq 1*})}=\\
& =-(v^\psi(s),{L_0^*} u^f(t))+(L
v^\psi(s)-\psi(s),u^f(t))\overset{\text{by}\,\,L\subset{L_0^*}}=\\
& =-(v^\psi(s),{L_0^*} u^f(t))+({L_0^*} v^\psi(s),u^f(t))
-(\psi(s),u^f(t))\overset{(\ref{Eq Green})}=\\
& =(\Gamma_1 v^\psi(s),\Gamma_2 u^f(t))-(\Gamma_2
v^\psi(s),\Gamma_1 u^f(t))-(\psi(s),u^f(t))\overset{(\ref{Eq
3}),\,(\ref{Eq Gamma1 v^psi=0}))}=\\
& =-(\Gamma_2 v^\psi(s),f(t))-(\psi(s),u^f(t))=:F(s,t)\qquad {\rm
in}\,\,\,\mathbb R_+\times\mathbb R_+.
\end{align*}
In the mean time, (\ref{Eq beta 2}) provides the zero Cauchy data:
$b\big|_{t=0}=b_t\big|_{t=0}=0$ on the bottom
$[s-t,s+t]\subset\overline{\mathbb R_+}$ of the cone $C_{s,t}$.
Integrating by D'Alembert formula, we get
$$
b(s,t)=-{1\over 2}\int_{C_{s,t}}F(\xi,\eta)\,d\xi\,d\eta
$$
which is (\ref{Eq DAlembert Cone C}).
\end{proof}

Now, fix $(s,t)$ provided $C_{s,t}\subset C_{{\sigma\over
2},{\sigma\over 2}}$. By this choice, in the cone $C_{s,t}$ we
have
\begin{equation}\label{Eq aux 44}
\Gamma_2 v^\psi(\xi)\big|_{\xi\leqslant\sigma}\overset{(\ref{Eq
(Gamma2 psi=0)})}=0
\end{equation}
and $(\psi(\xi),u^f(\eta))=0$, the latter being valid in view of
$u^f(\eta)\in {\mathscr U}^{\eta}\subset{\mathscr U}^{\sigma\over
2}\subset{\mathscr U}^\sigma$, whereas $\psi(\xi)$ is orthogonal
to ${\mathscr U}^\sigma$. Thus, both summands under integral in
(\ref{Eq DAlembert Cone C}) vanish in the cone and we get
$(v^\psi(s),u^f(t))=0$. Since $f\in{\mathscr M}$ is arbitrary, the
last equality means that $v^\psi(s)\bot {\mathscr U}^t$ holds.
Keeping $s$ fixed and varying $t\in[0,\sigma-s]$ (until $(s,t)\in
C_{{\sigma\over 2},{\sigma\over 2}}$ holds), we get $v^\psi(s)\bot
{\mathscr U}^{\sigma-s}$. Varying $s$ in the admissible segment
$[0,{{\sigma}\over 2}]$, we arrive at
\begin{equation}\label{Eq aux 3 sigma/2 first}
v^\psi(s)\in\overline{{\mathscr U}}\ominus\overline{{\mathscr
U}^{\sigma-s}},\qquad 0\leqslant s\leqslant {{\sigma}\over 2}\,.
\end{equation}

\noindent{\bf Step 2.}\,\,\,In the above considerations, to extend
the segment to $[0,\sigma]$ is not possible since for $s<t$ the
bottom $[s-t,s+t]$ of the cone $C_{t,s}$ does not fit in
$\overline{\mathbb R_+}$. Therefore, we change the cone for
$$
C'_{s,t}:=\{(\xi,\eta)\in \overline{\mathbb
R^2_+}\,|\,\,0\leqslant \xi\leqslant
s,\,\,\xi-s+t\leqslant\eta\leqslant -\xi+s+t\}\,.
$$

Fix $(s,t)$ provided $C'_{s,t}\subset C'_{{\sigma\over
2},{\sigma\over 2}}$. Repeating the same calculations as on Step
1, we arrive at the Cauchy problem for the string equation
\begin{align*}
& b_{tt}-b_{ss}=F(s,t), \qquad 0<s<t;\\
& b\big|_{s=0}=b_s|_{s=0}\overset{(\ref{Eq beta 2})}=0,\qquad
t\geqslant 0
\end{align*}
for the same $b$ and $F$ as before. Integrating by D'Alembert, we
get
\begin{equation*}
(v^\psi(s),u^f(t))=-{1\over 2}\int_{C'_{s,t}}\left[(\Gamma_2
v^\psi(\xi),f(\eta))+(\psi(\xi),u^f(\eta))\right]\,d\xi\,d\eta,\quad
0\leqslant s\leqslant t.
\end{equation*}
By (\ref{Eq aux 44}) and orthogonality $(\psi(\xi),u^f(\eta))=0$
for all $\xi<\sigma$, the summands in the integral vanish and we
obtain $(u^f(t),v^\psi(s))=0$ for $(s,t)\in C'_{{\sigma\over
2},{\sigma\over 2}}$. Keeping ${\sigma\over 2}<t<\sigma$ fixed and
extending $s$ from 0 to $\sigma-t$, we conclude that
$(v^\psi(\sigma-t),u^f(t))=0$ holds for $0\leqslant
t\leqslant\sigma$. Since $f\in {\mathscr M}$ is arbitrary, the
latter obeys $v^\psi(t)\bot\overline{{\mathscr U}^{\sigma-t}}$,
i.e.,
\begin{equation}\label{Eq aux 3 sigma/2 second}
v^\psi(t)\in\overline{{\mathscr U}}\ominus\overline{{\mathscr
U}^{\sigma-t}},\qquad {\sigma\over 2}\leqslant t\leqslant
{{\sigma}}\,.
\end{equation}
\smallskip

Putting (\ref{Eq aux 3 sigma/2 first}) and (\ref{Eq aux 3 sigma/2
second}) together, we obtain
\begin{equation}\label{Eq sigma-t+sigma}
v^\psi(t)\in\overline{{\mathscr U}}\ominus\overline{{\mathscr
U}^{\sigma-t}},\qquad 0\leqslant t\leqslant \sigma
\end{equation}
and establish the first part of the Theorem for smooth $\psi$.
Approximating $\psi\in\Psi$ with smooth sources and using the
continuity of the map $\psi\big|_{[0,t]}\to v^\psi(t)$, we
complete the proof of the first part.

It remains to prove the relation $v^\psi(t)\in\overline{{\mathscr
U}^{\tau+t}}$.
\smallskip

\noindent{\bf Step 3.}\,\,\,Let $I^t:{\mathscr H}\to{\mathscr
H}$,\,\,$I^ty:=v^y(t)$ be a map that resolves problem (\ref{Eq
beta 1'})--(\ref{Eq beta 2'}). By (\ref{Eq v^delta y repres
general}) we have $I^t=L^{-{1\over 2}}\sin [tL^{{1\over 2}}]$, so
that $I^t$ is a bounded self-adjoint operator in ${\mathscr H}$.
Let us establish the following.
\begin{Lemma}\label{L I tU}
The relation
\begin{equation}\label{Eq I tU}
I^t\overline{{\mathscr U}^\tau}\subset\overline{{\mathscr
U}^{\tau+t}},\qquad \tau>0,\,\,t>0
\end{equation}
holds.
\end{Lemma}
\begin{proof} Take an $f\in\mathscr M$ and $\theta>0$.
By virtue (\ref{Eq sigma-t+sigma}), for a source $\psi\in\Psi$
acting from $\overline{\mathscr U}\ominus\overline{\mathscr
U^{\tau+\theta}}$ one has
$$
(v^\psi(\tau+\theta-t), u^f(t))=0,\qquad 0\leqslant t\leqslant
\theta+\tau.
$$
Putting $t=\tau$, we have
\begin{align*}
& 0=(v^\psi(\theta), u^f(\tau))\overset{(\ref{Eq v^psi repres
general})}=(\int_0^\theta L^{-{1\over 2}}\sin[(\theta-s)L^{{1\over
2}}]\psi(s)\,ds, u^f(\theta)\,)=\\
& = \int_0^\theta(\psi(s),L^{-{1\over 2}}\sin[(\theta-s)L^{{1\over
2}}]u^f(\tau)\,)\,ds.
\end{align*}
By arbitrariness of $\psi$, the latter leads to
\begin{equation*}
L^{-{1\over 2}}\sin[(\theta-s)L^{{1\over
2}}]u^f(\tau)\,\bot\,\overline{\mathscr
U}\ominus\overline{\mathscr U^{\tau+\theta}}, \qquad 0\leqslant
s\leqslant \theta,
\end{equation*}
that is equivalent to $L^{-{1\over 2}}\sin[(\theta-s)L^{{1\over
2}}]u^f(\tau)\in\overline{\mathscr U^{\tau+\theta}}$. Putting
$s=0$ and $\theta=t$, we get $I^t
u^f(\tau)\subset\overline{\mathscr U^{\tau+t}}$. Since the waves
$u^f(\tau)$ are dense in $\overline{\mathscr U^{\tau}}$, we arrive
at (\ref{Eq I tU}).
\end{proof}
\smallskip

Let a source $\psi$ satisfy $\psi(t)\in\overline{{\mathscr
U}^\tau}$ for all $t\geqslant 0$. Representing
$$
v^\psi(t)\overset{(\ref{Eq v^psi repres general})}=\int_0^t
\left[I^{t-s}\,\psi(s)\right]\,ds,\qquad t>0,
$$
and taking into account $I^{t-s}\psi(s)\overset{(\ref{Eq I
tU})}\in \overline{{\mathscr U}^{\tau+t}}$ for all $s\leqslant t$,
we conclude that $v^\psi(t)\in\overline{{\mathscr U}^{\tau+t}}$ is
valid and, thus, prove Theorem \ref{T1}.
\end{proof}

The crucial observation that interior products of waves satisfy
the string equation $\square\, b=F$ is due to
A.S.Blagoveshchenskii. It was it that made possible to develop a
version of the BC-method for solving dynamical (time-domain)
inverse problems \cite{Bel DAN'87}. Simple but productive trick,
which consists in changing the roles of the spatial variable $s$
and time $t$ (with the replacement of the cone $C_{s,t}$ by the
cone $C'_{s,t}$) was also invented by Aleksandr Sergeevich
\cite{BelBlag,Blag,Blag2}.

\section{Wave parts of systems and operators}

\subsubsection*{Systems $\beta_{\mathscr D}$ and $\beta_{\mathscr U}$}
\noindent$\bullet$\,\,\, Recall that system $\beta$ is of the form
\begin{align*}
& \ddot v+Lv = \psi  && {\rm in}\,\,\,{{\mathscr H}}, \,\,\,t>0;\\
& v|_{t=0}=\dot v|_{t=0}=0 && {\rm in}\,\,\,{{\mathscr H}},
\end{align*}
and its trajectory is
\begin{equation*}
v^\psi(t)=L^{-{1\over 2}}\int_0^t\sin[(t-s)L^{{1\over
2}}]\,\psi(s)\,ds,\qquad t>0.
\end{equation*}
As a partial case, we deal with the sources $\psi=\delta(t)y$ and
the corresponding trajectory $v=v^{\delta y}=:v^y$ of the system
\begin{align*}
& \ddot v+Lv = 0  && {\rm in}\,\,\,{{\mathscr H}}, \,\,\,t>0;\\
& v|_{t=0}=0,\,\,\, \dot v|_{t=0}=y && {\rm in}\,\,\,{{\mathscr
H}},
\end{align*}
which is represented by
\begin{equation*}
v^{y}(t)=L^{-{1\over 2}}\sin[(t-s)L^{{1\over 2}}]\,y,\qquad t>0.
\end{equation*}

\noindent$\bullet$\,\,\, Recall the decomposition ${\mathscr
H}=\overline{{\mathscr U}}\oplus{\mathscr D}$, where ${\mathscr
U}:={\rm span\,}\{{\mathscr U}^t\,|\,\,t>0\}$.
\begin{Lemma}\label{L evolution in D}
If $\psi(t)\in{\mathscr D}$\, ($\psi(t)\in{\mathscr U}$) holds for
$t>0$ then $v^\psi(t)\in{\mathscr D}$\, ($v^\psi(t)\in{\mathscr
U}$) is valid for all $t>0$. If $y\in{\mathscr
D}\,\,(y\in{\mathscr U}$) holds then $v^y(t)\in{\mathscr D}$\,
($v^y(t)\in{\mathscr U}$) is valid for all $t>0$.
\end{Lemma}
\begin{proof} $\bf 1.$\,\,\,Let $\psi(t)\in{\mathscr D}\cap
\mathscr N$,\,\,$t>0$. Since $\psi(t)\bot{\mathscr U}$ for $t>0$,
relation (\ref{Eq (f,Gamma2 v psi)}) easily imply $\Gamma_2
v^\psi\big|_{t>0}=0$. Hence, by (\ref{Eq DAlembert Cone C}) we
have $(u^f(t),v^\psi(s))=0$ for all $s,t$. Therefore,
$v^\psi(s)\bot{\mathscr U}$ holds for all $s>0$. By approximating,
if necessary, the source $\psi\in\mathscr D$ with the sources of
the class $\mathscr N$, one cancels the restriction
$\psi(t)\in{\mathscr D}\cap \mathscr N$.

As one can easily verify, the same is true for the sources
$\psi=\delta(t) y$ with $y\in{\mathscr D}$: we have
$v^y(t)\in{\mathscr D}$ for all $t>0$.
\smallskip

$\bf 2.$\,\,\,Let $\psi(t)\in{\mathscr U}$,\,\,$t>0$. Fix an
arbitrary $y\in{\mathscr D}$. By (\ref{Eq v^psi repres general})
we have
\begin{align*}
& (v^\psi(t),y)= \int_0^t\left(L^{-{1\over 2}}\sin[(t-s)L^{{1\over
2}}]\,\psi(s),y\right)\,ds=\\
& =\int_0^t\left(\psi(s),L^{-{1\over 2}}\sin[(t-s)L^{{1\over
2}}]\,y\right)\,ds=\int_0^t\left(\psi(s),v^y(t-s)\right)\,ds=\\
& =0,\qquad t>0,
\end{align*}
the latter equality being valid due to relation $v^y(t)\in
{\mathscr D}$ proved above. Thus, we arrive at $v^\psi(t)\bot y$,
i.e., $v^\psi(t)\in{\mathscr U},\,\,\,t>0$.
\end{proof}

\noindent$\bullet$\,\,\, As a result, we conclude that system
$\beta$ evolves either in the subspace ${\mathscr D}$ or in the
subspace ${\mathscr U}$, depending on the source $\psi$ acting
from ${\mathscr D}$ or ${\mathscr U}$ respectively. It means that
$\beta$ splits in two independent (noninteracting) systems
$\beta_{\mathscr D}$ and $\beta_{\mathscr U}$, the second system
sharing the common evolution space with DSBC $\alpha$. If $\alpha$
is controllable, i.e., $\overline{\mathscr U}={\mathscr H}$ holds,
then ${\mathscr D}=\{0\}$ and system $\beta_{\mathscr D}$ is
absent \footnote{It is the case in the Example.}. Recall that the
latter occurs if and only if operator $L_0$, which determines all
systems under consideration, is completely non-selfadjoint
\cite{BD_DSBC}. One can claim that system $\beta_{\mathscr D}$ is
a part of system $\beta$ uncontrollable (unobservable) from
boundary. This picture is in full agreement with the general
systems theory: see \cite{KFA}, Chapter 10.

\subsubsection*{Space and wave parts of ${L_0^*}$}

\noindent$\bullet$\,\,\, Fix $T>0$ and assume that operator
${L_0^*}$ has a part in $\overline{{\mathscr U}^T}$. Recall that
this means $ \overline{\overline{{\mathscr U}^T}\cap{\rm
Dom\,}{L_0^*}}=\overline{{\mathscr U}^T}$ and
${L_0^*}\,[\overline{{\mathscr U}^T}\cap{\rm
Dom\,}{L_0^*}]\subset\overline{{\mathscr U}^T}$. Simplifying the
notation, we denote the part ${L_0^*}_{\,\overline{{\mathscr
U}^T}}$ by ${L_0^*}^T$. This part is a closable operator in
$\overline{{\mathscr U}^T}$ and we preserve the same notation
${L_0^*}^T$ for its closure. We say ${L_0^*}^T$ to be {\it a space
part} of ${L_0^*}$ in $\overline{{\mathscr U}^T}$.

In the mean time, the lineal set ${\mathscr U}^T$ of smooth waves
is dense in $\overline{{\mathscr U}^T}$ and invariant:
${L_0^*}{\mathscr U}^T={\mathscr U}^T$ holds (see (\ref{Eq invar U
tau})). Therefore the operator
$$
{L_0^*}^T_u:\overline{{\mathscr U}^T}\to\overline{{\mathscr
U}^T},\,\,\,{\rm Dom\,}{L_0^*}^T_u={\mathscr
U}^T,\,\,\,{L_0^*}^T_u\, y:={L_0^*} y
$$
is well defined, densely defined and closable in
$\overline{{\mathscr U}^T}$. We preserve the same notation
${L_0^*}^T_u$ for its closure and call it {\it a wave part} of
${L_0^*}$ in $\overline{{\mathscr U}^T}$.

As is evident, ${L_0^*}^T_u\subset{L_0^*}^T$ holds but something
more can be said about relations between these parts. In the
following Lemma, by {\it isomorphism} we mean an injective,
surjective, bounded, and boundedly invertible operator. Recall
that ${\mathscr K}:={\rm Ker\,}{L_0^*}$. Denote $\mathscr
F^T:=L_2([0,T];\mathscr K)$ and $\mathscr
M^T:=\{f\big|_{[0,T]}\,|\,\,f\in\mathscr M\}$.
\begin{Lemma}\label{L space and wave parts}
Assume that $W^T$ is an isomorphism from ${\mathscr F^T}$ to
$\overline{{\mathscr U}^T}$ and ${\mathscr
K}\cap\overline{{\mathscr U}^T}=\{0\}$ holds. Then the space and
wave parts coincide: ${L_0^*}^T={L_0^*}^T_u$ holds.
\end{Lemma}
\begin{proof}
Choose a pair $(y,{L_0^*}^T y)=(y,{L_0^*} y)\in{\rm
graph\,}{L_0^*}^T$. In view of (\ref{Eq M}) and (\ref{Eq invar U
tau}),  one can find a sequence of controls $g_n \in {\mathscr
M}^T$ provided $u^{g_n}(T)\to {L_0^*} y$ in $\overline{{\mathscr
U}^T}$. By isomorphism of $W^T$, this sequence has to converge:
$g_n\to g$ in ${\mathscr F^T}$. Representing uniquely $g_n=-\ddot
f_n,\,\,\, g=-\ddot f$ with $f_n\in{\mathscr M}^T$, we have the
convergence $f_n\to f$ in ${\mathscr F^T}$, which implies
$u^{f_n}(T)\to u^f(T)$ in $\overline{{\mathscr U}^T}$. Along with
the latter convergence, one has ${L_0^*}
u^{f_n}(T)\overset{(\ref{Eq steady state+})}=-\ddot
u^{f_n}(T)=u^{-\ddot f_n}(T)=u^{g_n}(T)\to {L_0^*} y$. As a
result, we conclude that $(u^f(T),{L_0^*} y)\in{\rm
graph\,}{L_0^*}^T_u$ holds.

In the mean time, ${L_0^*}^T_u\subset{L_0^*}^T$ obeys ${\rm
graph\,}{L_0^*}^T_u\subset{\rm graph\,}{L_0^*}^T$. Hence, both
pairs $(y,{L_0^*} y)$ and $(u^f(T),{L_0^*} y)$ belong to ${\rm
graph\,}{L_0^*}^T$. The latter follows to $(y-u^f(T),0)\in{\rm
graph\,}{L_0^*}^T$, i.e., $y-u^f(T)\in {\mathscr K}$. In view of
${\mathscr K}\cap\overline{{\mathscr U}^T}=\{0\}$ we arrive at
$y=u^f(T)$ and conclude that the graphs of the space and wave
parts of ${L_0^*}$ in $\overline{{\mathscr U}^T}$ coincide, i.e.,
${L_0^*}^T={L_0^*}^T_u$ does hold.
\end{proof}
\smallskip

\noindent$\bullet$\,\,\, The assumption on $W^T$ to be an
isomorphism is rather restrictive: for instance, it is invalid in
the Example. However, analyzing the proof of Lemma \ref{L space
and wave parts}, it is easy to remark that such an assumption can
be relaxed as follows. It suffices to require the convergence of
${L_0^*} u^{f_n}(T)$ to imply the convergence of $u^{f_n}(T)$ in
$\overline{{\mathscr U}^T}$, whereas the convergence of $f_n$ in
${\mathscr F^T}$ is not necessary. As can be shown, the latter
holds in the Example for times $T<T_*$.

In this regard, it is worth noting that the convergence of
${L_0^*} u^{f_n}(T)$ implies the convergence of the summands
$u^{f_n}_L(T)\in{\rm Dom\,} L$ in representation (\ref{Eq u^f+}).
This reflects a general fact: in (\ref{Eq Vishik Decomp
components}), if ${L_0^*} y_n$ converges then $y'_n=L^{-1}{L_0^*}
y_n$ also converges.

The counterexamples of ${L_0^*}^T\not={L_0^*}^T_u$ are not known
and a hope for the equality with no assumptions is still alive.

\subsubsection*{Completeness of waves}

\noindent$\bullet$\,\,\, In system $\beta$ one can introduce a
'source--state' map ${\cal I}^t\psi:=v^\psi(t),\,\,t>0$ for
$\psi\in L^{\rm loc}_2(\mathbb R_+;{\mathscr H})$ \footnote{It is
used in \cite{B JOT,BSim_3 Mat Sbor} and called an {\it isotony}}.
Fix a subspace ${\mathscr A}\subset{\mathscr H}$ and denote by
$\Psi_{\mathscr A}:=  L^{\rm loc}_2(\mathbb R_+;{\mathscr A})$ the
space of sources acting from ${\mathscr A}$. In this notation, the
statement of Theorem \ref{T1} takes the form
$$
{\cal I}^t\Psi_{\overline{{\mathscr
U}^\tau}\ominus\overline{{\mathscr
U}^\sigma}}\subset\overline{{\mathscr
U}^{\tau+t}}\ominus\overline{{\mathscr U}^{\sigma-t}},\qquad
0<\sigma<\tau, \,\,t>0.
$$
In the mean time, in the Example, as well as in many other
applications, a stronger relation occurs: not embedding but
equality holds. It is interpreted as a completeness of waves in
domains, which they fill up. In the abstract case, by analogy with
applications, one may speak about completeness of waves in the
filled subspaces. Below we show a result of this kind under some
additional assumption.
\smallskip

\noindent$\bullet$\,\,\, Let $P^{\epsilon}$ be the projection in
$\overline{{\mathscr U}}$ onto $\overline{{\mathscr
U}^{\epsilon}}$; denote $P^{\epsilon}_\bot:=\mathbb
I-P^{\epsilon}$. Assume that there is a continuous (in norm)
family of the bounded operators
$N^{\epsilon},\,\,\,0\leqslant{\epsilon}\leqslant{\epsilon}_*$
such that
\begin{equation}\label{Eq neutrol 1}
N^0=\mathbb I;\qquad N^{\epsilon} y\in {\rm Dom\,}
L,\,\,\,P^{\epsilon}_\bot N^{\epsilon} y=P^{\epsilon}_\bot
y,\qquad y\in{\rm Dom\,} L^*_0
\end{equation}
holds. Note that, by (\ref{Eq neutrol 1}), one has $y-N^{\epsilon}
y\in \overline{{\mathscr U}^{\epsilon}}$.

In the Example, in capacity of $N^{\epsilon}$ one can take the
multiplication by a smooth function $\chi^{\epsilon}$ provided
$0\leqslant\chi^{\epsilon}(\cdot)\leqslant
1,\,\,\,\chi^{\epsilon}\big|_{\Omega\setminus\Omega^{\epsilon}}=1,\,\,\,\chi^{\epsilon}\big|_\Gamma=0$.
Parameter ${\epsilon}_*$ is chosen to provide the interior
boundary of the subdomain $\Omega^{{\epsilon}_*}$ to be smooth. By
analogy to $\chi^{\epsilon}$ we call $N^{\epsilon}$ a {\it
neutralizer}.
\begin{Lemma}\label{L completeness}
Let operator  ${L_0}$ be such that the neutralizers
$N^{\epsilon},\,\,0\leqslant{\epsilon}\leqslant{\epsilon}_*$ do
exist and let operator ${L_0^*}$ have a space part
${L_0^*}^{\epsilon}$ in each subspace $\overline{{\mathscr
U}^{\epsilon}}$. Then the relation
\begin{equation}\label{Eq completeness}
\overline{{\cal I}^t\Psi_{\overline{{\mathscr
U}^\tau}}}=\overline{{\mathscr U}^{\tau+t}},\qquad \tau>0, \,\,t>0
\end{equation}
is valid.
\end{Lemma}
\begin{proof}
\,\,\,({\it sketch})\,\,\,Fix\, $0<\tau<T$ and take $f\in{\mathscr
M}^T$. The corresponding waves  $u^f(T)$ constitute a dense set in
$\overline{{\mathscr U}^T}$ by definition of the latter. As is
evident, the projections $P^{\epsilon}_\bot u^f(T)$ are dense in
$\overline{{\mathscr U}^T}\ominus\overline{{\mathscr
U}^{\epsilon}}$. Loosely speaking, the idea of the proof is to
represent $u^f(T)$ as a wave produced by a relevant source $F$,
which acts from the subspace $\overline{{\mathscr U}^\tau}$.

The wave $u^f$ that satisfies (\ref{Eq 1})--(\ref{Eq 3}), is also
determined by the system
\begin{align*}
& \ddot u+{L_0^*} u= \delta \dot u^f(\tau)+\dot\delta u^f(\tau),
&&
\tau<t<T;\\
& u\big|_{t=\tau}=\dot u\big|_{t=\tau}=0;\\
& \Gamma_1 u= f, && \tau\leqslant t\leqslant T,
\end{align*}
where $\delta=\delta(t)$ is the Dirac function.

Taking ${\epsilon}<\tau$ and representing $u^f=N^{\epsilon} u^f+
[u^f-N^{\epsilon} u^f]$, we get the system $\beta$ (with a shifted
time) of the form
\begin{align}
\label{Eq syst complete 1} &  \ddot{N^{\epsilon}
u^f}+L\,N^{\epsilon} u^f=F^{\epsilon}, &&
\tau<t<T;\\
\label{Eq syst complete 2} & {N^{\epsilon}
u^f}\big|_{t=\tau}=\dot{N^{\epsilon} u^f}\big|_{t=\tau}=0 &&
\end{align}
(we use ${L_0^*} N^{\epsilon} u^f=LN^{\epsilon} u^f$) with a
source
$$
F^{\epsilon}(t):=-\left[\frac{d^2}{dt^2}+{L_0^*}\right](u^f(t)-N^{\epsilon}
u^f(t))+ \delta \dot u^f(\tau)+\dot\delta u^f(\tau),
$$
where $u^f(t)-N^{\epsilon} u^f(t)\in {\rm Dom\,}
{L_0^*}\cap\overline{{\mathscr U}^{\epsilon}}$ holds. By the
latter, we have ${L_0^*}[u^f(t)-N^{\epsilon}
u^f(t)]={L_0^*}^{\epsilon}[u^f(t)-N^{\epsilon} u^f(t)]\in
\overline{{\mathscr U}^{\epsilon}}\subset \overline{{\mathscr
U}^\tau}$, where $L_0^{* {\epsilon}}$ is the space part of
${L_0^*}$ in $\overline{{\mathscr U}^{\epsilon}}$. In the mean
time, $u^f(\tau)$ and $\dot u^f(\tau)$ belong to
$\overline{{\mathscr U}^\tau}$. So, the source $F^{\epsilon}$ does
act from the subspace $\overline{{\mathscr U}^\tau}$, its time of
acting being equal to $T-\tau$.

By the latter, shifting time $t\to t-\tau$ in (\ref{Eq syst
complete 1})--(\ref{Eq syst complete 2}) and applying Theorem
\ref{T1}, we see that the source $F^{\epsilon}$ produces the wave
$v^{F^{\epsilon}}(T-\tau)=N^{\epsilon} u^f(T)$. When $f$ varies in
${\mathscr M}^T$, the projections $P^{\epsilon}_\bot N^{\epsilon}
u^f(T)=P^{\epsilon}_\bot u^f(T)$ of such waves constitute a
complete system in $\overline{{\mathscr
U}^T}\ominus\overline{{\mathscr U}^{\epsilon}}$. Tending
${\epsilon}\to 0$, by (\ref{Eq neutrol 1}) we conclude that there
is a sequence $\{F^{\epsilon}\}$ of the sources, which act from
$\overline{{\mathscr U}^\tau}$ and provide
$v^{F^{\epsilon}}(T-\tau)\to u^f(T)$. Therefore, these sources
produce a system of waves complete in $\overline{{\mathscr U}^T}$.

Since $\tau$ and $T$ are arbitrary, it is easy to see that what
has been proved is equivalent to the equality (\ref{Eq
completeness}). To justify the formal operations with $\delta$ and
$\dot\delta$, one needs to approximate them by a proper smooth
regularizations: see, e.g, \cite{BD_DSBC}.
\end{proof}

The idea to use a neutralizer comes from the Example, where its
existence is guaranteed and do not require additional assumptions.

\subsubsection*{One more abstract property}
Here is one more fact that takes place in the Example, which can
be generalized. At first glance, it looks very specific but, as
will be shown, does have an abstract analog.

Recall that the Friedrichs extension $L=-\Delta$ of the minimal
Laplacian is defined on ${\rm Dom\,} L=H^2(\Omega)\cap
H^1_0(\Omega)$. Let $y\in{\rm Dom\,} L$ and ${\rm
supp\,}y\subset\Omega\setminus\Omega^\tau$ for a positive
$\tau<T_*$, so that ${{\rm supp\,}\,}y$ is separated from the
boundary $\Gamma$ by the distance $\tau$. In such a case, we have
$y\big|_\Gamma=\partial_\nu y\big|_\Gamma=0$ and hence $y\in
H^2_0(\Omega)$ i.e., $y\in{\rm Dom\,} L_0$ holds.
\begin{Lemma}\label{L y=y0}
Let $\tau>0$ satisfy $\overline{{\mathscr
U}}\ominus\overline{{\mathscr U}^\tau}\not=\{0\}$ and $y\in
[\overline{{\mathscr U}}\ominus\overline{{\mathscr
U}^\tau}]\cap{\rm Dom\,} L$ hold. Then the relation $y\in{\rm
Dom\,} L_0$ is valid.
\end{Lemma}
\begin{proof}
Recall that $y\overset{(\ref{Eq L0,L})}\in {\rm Ker\,}\Gamma_1$.
Let $T>\tau$. By $y\in{\rm Dom\,} L$ and (\ref{Eq (Gamma2
v^y=0)}), we have $\Gamma_2 v^y\big|_{0< t\leqslant\tau}=0$. Since
$y\in {\rm Dom\,} L$, one has
$$
\dot v^y(t)\overset{(\ref{Eq v^delta y repres
general})}=\cos[tL^{1\over 2}]\,y=L^{-1}\cos[tL^{1\over
2}]\,Ly,\qquad t\geqslant 0.
$$
This implies $\dot v^y\in C([0,T];{\rm Dom\,} L)$, where ${\rm
Dom\,} L$ is endowed with the $L$-graph norm \cite{BirSol}. By
corresponding continuity of $\Gamma_{1,2}$, we get $\Gamma_2 \dot
v^y\big|_{t=+0}=\Gamma_2y=0$. So, $y\in {\rm
Ker\,}\Gamma_1\cap{\rm Ker\,}\Gamma_2$, i.e., $y\overset{(\ref{Eq
L0,L})}\in{\rm Dom\,} L_0$ does hold.
\end{proof}

\subsubsection*{A bit of philosophy}
A character and goal of this paper may be commented on as follows.
In our opinion, working in specific branches of mathematical
physics  (like inverse problems), it is however reasonable to pay
attention to abstractions. Let us refer to the authority of
classicists. According to Van der Waerden, a maxima, which Emmy
Noether adhered to in her work, claims that {\it any
interconnection between numbers, functions and operations becomes
transparent, available for further generalization and productive
only after that, as it is separated from any specific objects and
is reduced to general terms}.

Systems $\alpha$ and $\beta$ are the general terms. We try to
follow the maxima.

\end{document}